\documentclass[10pt,twoside]{amsart}
\usepackage{geometry}
\usepackage[english]{babel}
\usepackage{graphicx}
\usepackage{amsmath}
\usepackage{amsfonts}

\begin{document}

\numberwithin{equation}{section}
\newtheorem{teo}{Theorem}
\newtheorem{lemma}{Lemma}
\newtheorem{defi}{Definition}
\newtheorem{coro}{Corollary}
\newtheorem{prop}{Proposition}
\numberwithin{lemma}{section}
\numberwithin{prop}{section}
\numberwithin{teo}{section}
\numberwithin{defi}{section}
\numberwithin{coro}{section}
\numberwithin{figure}{section}
\title{Curvature Motion in Time-dependent Minkowski Planes}
\author{Vitor Balestro}
 
\address{ Instituto de Matem\'{a}tica -- UFF --
Niter\'{o}i -- Brazil}
\email{vitorbalestro@id.uff.br}

\begin{abstract}
In this paper we study a flow by minkowskian curvature where we have a different Minkowski plane at each time. We derive some evolution formulas, present sufficient hypotesis for the short time existence and convexity of solutions and study the motion considering a particular type of families of Minkowski norms. Also, as a corollary we establish a result about a certain family of nonlinear parabolic PDE's.
\end{abstract}

\maketitle
\section{Introduction}
The idea behind the Minkowski curvature flow treated in \cite{mink} is to consider the plane $\mathbb{R}^2$ with a different norm and study the motion of curves evolving by the minkowskian curvature in this new context. But, since we are working with (almost) arbitrary norms in $\mathbb{R}^2$ is quite natural to ask what happens if the norms change along the motion. There is a (maybe naive) physical application for this: one can imagine a motion in an ambient where the "resistence" varies with time. We will do our generalization as follows: consider $\mathbb{R}^2\times [0,T)$ "sliced" in the following way: each plane $\mathbb{R}^2\times\{t\}$ is identified with the Minkowski plane $\mathbb{R}^2$ endowed with a norm given by a $\mathcal{P}_t$-unit ball with boundary parameterized as usual by $p_t(\theta)$. The idea is to consider each curve of the flow lying in one of these planes and evolving according to its geometry. We could also think about this as a motion in a 2-dimensional fibration over an interval. In this paper we study some results concerning the existence and convergence of the flow under certain conditions to the family of norms. In section two we define the flow, derive some evolution formulas and give some conditions on the family of norms that guarantee desired properties for the flow. In section 3 we study a particular case for what we can ensure existence of the solution until the (usual) area converges to zero. To finish the paper we give, using a geometric argument, a estimate for a uniform blow up time for the solutions of a family of parabolic PDE's. \\

\section{The time-dependent minkowskian curvature flow}
Let $a:S^1\times[0,T)\rightarrow \mathbb{R}$ a strictly positive smooth function which is $\pi$-periodic on $\theta$ for every fixed $t$. For a function $f(\theta,t)$ we will denote the spatial derivatives by $f'$ and the time derivative by $\dot{f}$. Suppose that we have $a(\theta,t)+a''(\theta,t) > 0$ everywhere. We can define a family of $\mathcal{P}$-unit balls on $\mathbb{R}^2$ defining $\mathcal{P}_t$ to be the region enclosed by the curve: \\
\[p_t(\theta) = p(\theta,t) = a(\theta,t)e_r + a'(\theta,t)e_\theta, \] \\
where $e_r = (\cos\theta, \sin\theta)$ and $e_\theta = (-\sin\theta,\cos\theta)$, as usual. For each $t$ we have also the dual ball $\mathcal{Q}_t$ to $\mathcal{P}_t$. For every $t$ the following holds: \\
\[ q(\theta,t) = \frac{p'(\theta,t)}{[p(\theta,t),p'(\theta,t)]}, \ \ \mathrm{and} \] \\
\[ p(\theta,t) = - \frac{q'(\theta,t)}{[q(\theta,t),q'(\theta,t)]} \] \\

We define the time-dependent minkowskian curvature flow (that, for sake of convenience, will be called "generalized flow", in constrast with the "usual flow" studied in \cite{mink}) associated to the family of norms given by $\mathcal{P}_t$ as an application $F:S^1\times[0,T_0)\rightarrow\mathbb{R}^2$, $T_0 \leq T$, which satisfies \\
\[ \frac{\partial F}{\partial u}(u,t) = v(u,t).q(\theta(u,t),t); \ \ \mathrm{and} \] \\
\[ \frac{\partial F}{\partial t}(u,t) = -k(u,t).p(\theta(u,t),t), \]\\
where $k(u,t)$ is the minkowskian curvature of the curve $u \mapsto F_t(u)$ calculated with respect to the $\mathcal{P}_t$ norm, and $\theta(u,t) + \pi/2$ is the angle between $\partial F/\partial u$ and the $x$-axis. Thus, we can write: \\
\[ \frac{\partial F}{\partial\theta}(u,t) = \lambda(u,t).q(\theta(u,t),t) \] \\
and then, the curvature is given by \\
\[ k(u,t) = \frac{[p(\theta(u,t),t),p'(\theta(u,t),t)]}{\lambda(u,t)} \] \\
Notice that we have $\lambda d\theta = vdu$. We derive now some evolution formulas. \\

\begin{lemma} We have, for every $(u,t)\in S^1\times[0,T)$: \\
\[\displaystyle\frac{\partial v}{\partial t} = -k^2v + \frac{\partial\log(a)}{\partial t}v\] \\

\end{lemma}
\begin{proof}
First, we compute \\
\[ \frac{\partial}{\partial t}\left(\frac{\partial F}{\partial u}\right) = \frac{\partial v}{\partial t}q + v\left(\frac{\partial q}{\partial\theta}\frac{\partial\theta}{\partial t} + \frac{\partial q}{\partial t}\right) \] \\ \\
Is easy to check that $\displaystyle\frac{\partial q}{\partial t} = -\frac{\partial\log(a)}{\partial t}q$. Since $\partial q/\partial\theta$ points on the $p$ direction, writing  $\displaystyle\frac{\partial}{\partial t}\left(\frac{\partial F}{\partial u}\right)$ in the basis $\{p,q\}$ the coefficient of $q$ is $\displaystyle\frac{\partial v}{\partial t}-\frac{\partial\log(a)}{\partial t}v$. Let us now compute \\
\[ \frac{\partial}{\partial u}\left(\frac{\partial F}{\partial t}\right) = -\frac{\partial k}{\partial u}p - k\frac{\partial p}{\partial\theta}\frac{\partial\theta}{\partial u} =  -\frac{\partial k}{\partial u}p -k^2vq \] \\

And then, since $u$ and $t$ are independent parameters we have the desired equality. \\
\end{proof}
Notice that we can still work with a $\mathcal{Q}_t$-arclength parameter $s_t$, but only for a fixed time. As a corollary of the previous lemma we have an evolution formula to the $\mathcal{Q}_t$-arclength of the curves as follows: \\
\[ \frac{\partial L_{\mathcal{Q}_{t}}}{\partial t} = \int_0^{2\pi}\frac{\partial v}{\partial t}\ du = -\int_0^{2\pi}k^2v \ du + \int_0^{2\pi}\frac{\partial\log(a)}{\partial t}v \ du, \] \\
\begin{lemma} \label{qlenght} If there exists $M \in \mathbb{R}$ such that $\dot{a}(\theta,t) \leq Ma(\theta,t)$ in $S^1\times[0,T)$, then the $\mathcal{Q}_t$-lenght of the curves is bounded along the motion.\\
\end{lemma}
\begin{proof} The hypotesis gives $\displaystyle\frac{\partial\log(a)}{\partial t} \leq M$ in $S^1\times [0,T)$. Then, we have \\
\[ \frac{\partial L_{\mathcal{Q}_t}}{\partial t} \leq \int_0^{2\pi}\frac{\partial\log(a)}{\partial t}v \ du \leq ML_{\mathcal{Q}_t}. \] \\
So, Gronwall's inequality yields $L_{\mathcal{Q}_t} \leq e^{MT}L_{\mathcal{Q}_0}$. \\
\end{proof}
We can also derive an evolution formula for the area $A(t)$ enclosed by the curves: \\
\begin{eqnarray*} \frac{\partial A}{\partial t} = \frac{\partial}{\partial t}\left(\frac{1}{2}\int_0^{2\pi}\left[F(u,t),\frac{\partial F}{\partial u}(u,t)\right]\ du\right) = \\ \\ = \frac{1}{2}\int_0^{2\pi}\left[\frac{\partial F}{\partial t}(u,t),\frac{\partial F}{\partial u}(u,t)\right]\ du + \frac{1}{2}\int_0^{2\pi}\left[F(u,t),\frac{\partial^2 F}{\partial t\partial u}(u,t)\right]\ du = \\ \\ = \int_0^{2\pi}-kv \ du = -\int_0^{L_{\mathcal{Q}_t}}k \ ds_t = -2A(\mathcal{P}_t)\end{eqnarray*} \\
Now, is easy to see that the evolution of the isoperimetric ratio is given by \\
\[ \frac{\partial}{\partial t}\left(\frac{L^2_{\mathcal{Q}_t}}{A(t)}\right) = -\frac{2L_{\mathcal{Q}_t}}{A(t)}\left(\int_0^{L_{\mathcal{Q}_t}}k^2 ds_t - A(\mathcal{P}_t)\frac{L_{\mathcal{Q}_t}}{A(t)}\right) + \frac{2L_{\mathcal{Q}_t}}{A(t)}\int_0^{L_{\mathcal{Q}_t}}\frac{\partial\log(a)}{\partial t} \ ds_t \] \\
Here we observe that the isoperimetric ratio may be an increasing function. Moreover, the evolution of the isoperimetric ratio depends on the choice of the considered family of norms. \\ \\
Let us now turn our attention to the existence of such a generalized flow. For a fixed $t$ we know that a $2\pi$-periodic, positive and $C^1$ function $k:[0,2\pi]\rightarrow \mathbb{R}$ is the $t$-minkowskian curvature of a simple, closed, strictly convex and $C^2$ curve if and only if the equalities \\
\begin{eqnarray} \int_0^{2\pi}\frac{a(\theta,t)+a''(\theta,t)}{k(\theta)}\sin(\theta) \ d\theta =  \int_0^{2\pi}\frac{a(\theta,t)+a''(\theta,t)}{k(\theta)}\cos(\theta) \ d\theta = 0 \end{eqnarray} \\
hold. With this in mind we claim that we don't need too strong hypotesis to ensure short-time existence for the generalized flow. This is justified in the next theorem.

\begin{teo} Consider a function $k:S^1\times [0,T_0)\rightarrow \mathbb{R}$, $T_0 \leq T$, such that $k\in C^{2+\alpha,1+\alpha}(S^1\times[0,T_0-\epsilon])$ for all $\epsilon > 0$, satisfying the evolution equation: \\
\begin{eqnarray} \frac{\partial k}{\partial t} = \frac{a}{a+a''}k^2k'' + \frac{2a'}{a+a''}k^2k' + k^3 + \frac{\partial\log(a+a'')}{\partial t}k \end{eqnarray} \\
with initial condition $k(\theta,0) = \varphi(\theta)$, where $\varphi$ is a strictly positive $C^{1+\alpha}$ function satistying (1). Assume that $t \mapsto a(\theta,t) + a''(\theta,t)$ is nondecreasing for each $\theta \in S^1$. Then, using this function (whose short term existence is guaranteed by the standard theory on parabolic equations) one can build the family of curves on parameter t: \\
\begin{eqnarray*} F(\theta,t) = \left( -\int_0^{\theta}\frac{a(\sigma,t)+a''(\sigma,t)}{k(\sigma,t)}\sin\sigma \ d\sigma - \int_0^ta(0,s)k(0,s) \ ds,\right. \\
\left.\int_0^{\theta}\frac{a(\sigma,t)+a''(\sigma,t)}{k(\sigma,t)}\cos\sigma \ d\sigma - \int_0^ta(0,s)k'(0,s)+a'(0,s)k(0,s) \ ds \right) \\ \end{eqnarray*} 
for which the following holds: \\
\begin{description}
\item[(a)] for each fixed $t$ the map $t \mapsto F(\theta,t)$ is a simple, closed and strictly convex curve parameterized as usual whose $t$-minkowskian curvature is given by $\theta \mapsto k(\theta,t)$ \\

\item[(b)] $\displaystyle\frac{\partial F}{\partial\theta}(\theta,t) = -k(\theta,t).p(\theta,t) - a(\theta,t)^2k'(\theta,t).q(\theta,t)$ \\
\end{description}
\end{teo}
\begin{proof}
The proofs of \textbf{(b)} and of \\
\[\int_0^{2\pi}\frac{a(\sigma,t)+a''(\sigma,t)}{k(\sigma,t)}\sin\sigma \ d\sigma = \int_0^{2\pi}\frac{a(\sigma,t)+a''(\sigma,t)}{k(\sigma,t)}\cos\sigma \ d\sigma = 0 \] \\
for every $t$ are straightfoward calculations just  like in the usual minkowskian curvature flow. Thus, we only need to prove that $k$ is strictly positive in $S^1\times [0,T_0)$. But the hypotesis on $a + a''$ guarantees that we have \\
\[ \frac{\partial\log(a+a'')}{\partial t} \geq 0 \ \ \mathrm{in} \ S^1\times [0,T_0) \] \\
and then we can repeat the proof of the usual case to prove that $k_{\mathrm{MIN}}(t)$ is bounded by below by $k_{\mathrm{MIN}}(0)$. This concludes the proof. \\
\end{proof}
From the evolution formulas we can see that the area $A(t)$ enclosed by the curve at time $t$ is a decreasing function. But the decay ratio depends on $A(\mathcal{P}_t)$, and then we cannot readily say that the area converges to $0$ even for infinite time. Furthermore, we don't even can tell if the $\mathcal{Q}_t$-lengths have an upper bound. But, working on a certain class of $\mathcal{P}_t$ families we can maybe give good answers to these questions. We work with a particular case in the next section. \\
\section{Homothetic family of $\mathcal{P}$-balls}
Consider a Minkowski norm on $\mathbb{R}^2$ given by the set $\mathcal{P} = \mathcal{P}_0$, whose boundary is parameterized by \\
\[ p_0(\theta) = a_0(\theta)e_r + a_0'(\theta)e_{\theta} \] \\
where $a_0$ is a $C^{\infty}$ and $\pi$-periodic function. Let $f:[0,\infty) \rightarrow \mathbb{R}$ be a $C^{\infty}$, positive and nondecreasing function such that $f(0) = 1$. Then, we can define a family of Minkowski norms putting \\
\[ a(\theta,t) = f(t)a_0(\theta) \] \\
and, naturally, $p(\theta,t) = a(\theta,t)e_r + a'(\theta,t)e_{\theta} = f(t)p_0(\theta)$. \\

\noindent\textit{Remark.} Even if this looks like a simple reparametrization at the time we point out that, here, we are calculating the curvatures in a different way at each time. This interpretation considers the flow evolving with respect to the geometry of each space. \\

The hypotesis that $f$ is nondecreasing guarantees short term existence of the generalized flow associated to this family of norms, taking by initial conditon a closed, strictly convex and smooth curve. In this particular case the evolution equation becames \\
\[ \frac{\partial k}{\partial t} = \frac{a_0}{a_0+a_0''}k^2k'' + \frac{2a_0'}{a_0+a_0''}k^2k' + k^3 + \frac{\dot{f}(t)}{f(t)}k, \] \\ We claim that in this case, if the solution continues until the area enclosed by the curves goes to 0, then the area converges to 0 in finite time. In fact, notice first that \\
\[ A(\mathcal{P}_t) = \int_0^{2\pi}[p(\theta,t),p'(\theta,t)] \ d\theta = f(t)^2\int_0^{2\pi}[p_0(\theta),p_0'(\theta)]\ d\theta = f(t)^2A(\mathcal{P}_0) \] \\
Now, the evolution formula for the area guarantees that \\
\[ A(t) = A(0) - 2\int_0^tA(\mathcal{P}_s) \ ds = A(0) - 2\int_0^tf(s)^2A(\mathcal{P}_0)\ ds \leq A(0) - 2tA(\mathcal{P}_0),  \] \\
since $f(0) = 1$ and $f$ is nondecreasing. Then, $A(t)$ converges to zero for some time $T_1 \leq A(0)/A(\mathcal{P}_0)$. \\

Our next claim is that we have an upper bound for the $\mathcal{Q}_t$-length of the curves in finite time. In fact, for finite time we can take an upper bound for $\displaystyle\frac{\dot{f}(t)}{f(t)}$ and use Lemma \ref{qlenght}. This proves that the $\mathcal{Q}_t$-lenght doesn't blow up along the motion. \\

We will now prove that in this case we also have solution until the area converges to 0. The strategy is basically the same that in the usual case treated in \cite{mink}. We consider a solution $k$ of (2) defined on $S^1\times [0,T)$ such that the area enclosed by the associated curves remains bounded away from zero (i.e., we have $T \leq T_1$) and prove that $k$ and its derivatives are bounded in $S^1\times [0,T)$. Finally, we use Ascoli-Arzela's theorem to extend $k$ past $T$. The difference here is that we have to deal with bounds to $f$ and its derivatives. This shouldn't be a problem since we are working with a finite time interval $[0,T) \subseteq [0,T_1]$, and then  compacity and the smoothness of $f$ guarantee the needed bounds. \\

Recall that in the minkowski plane we still can define the median curvature of a curve parameterized by the usual $\theta$ as the supremum of the values $x$ for which we have $k(\theta) > x$ in some interval of length $\pi$. We have the estimate \\
\[ k^* \leq C\frac{L_{\mathcal{Q}}}{A} \] \\
for a constant $C$ given by \\
\[ C = \left(\max_{\theta\in [0,2\pi]}|q(\theta)|\right)^2\max_{\theta\in [0,2\pi]}[p(\theta),p'(\theta)] \] \\
And then, in our generalized case we have \\
\[ k^*(t) \leq C(t)\frac{L_{\mathcal{Q}_t}}{A(t)}, \] \\
with \\
\[C(t) = \left(\max_{\theta\in [0,2\pi]}\left| \frac{1}{a(\theta,t)}\right|\right)^2\max_{\theta \in [0,2\pi]}\left[p(\theta,t),p'(\theta,t)\right] \] \\
The point here is that there is an uniform upper bound for the median curvature along the motion if the area is bounded by below by a constant greater then zero. We already know that we have an upper bound for $L_{\mathcal{Q}_t}$. We also can (in such a natural way!) produce an uniform upper bound for $C(t)$ just rewriting  \\
\begin{eqnarray*} C(t) = \frac{1}{f(t)^2}\left(\max_{\theta\in [0,2\pi]}\frac{1}{a_0(\theta)}\right)^2f(t)^2\max_{\theta \in [0,2\pi]}[p_0(\theta),p_0'(\theta)] = \\ \\ = \left(\max_{\theta\in [0,2\pi]}\frac{1}{a_0(\theta)}\right)^2\max_{\theta \in [0,2\pi]}[p_0(\theta),p_0'(\theta)]  \end{eqnarray*} \\
This is summarized as follows \\
\begin{lemma} If $A(t)$ is bounded away from zero on $[0,T)$, then there is an uniform upper bound for $k^*(t)$ in $[0,T)$. \\
\end{lemma}
By consequence we have the following proposition \\
\begin{prop}If $k^*(t)$ is bounded in $[0,T)$, then the function\\
\[ t \mapsto \displaystyle\int_0^{2\pi}(a_0(\theta)+a_0''(\theta))a_0(\theta)\log\left(\frac{k(\theta,t)}{f(t)}\right)\ d\theta \] \\
is also bounded in $[0,T)$. \\
\end{prop}
\begin{proof} A straightfoward calculation gives us \\
\begin{eqnarray*} \frac{d}{dt}\left(\int_0^{2\pi}(a_0(\theta)+a_0''(\theta))a_0(\theta)\log\left(\frac{k(\theta,t)}{f(t)}\right)\ d\theta\right) = \\ \\ = \int_0^{2\pi}(a_0(\theta)k(\theta,t))^2 - \left((a_0(\theta)k(\theta,t))'\right)^2d\theta \end{eqnarray*} \\
From now on we use basically the same strategy used in the usual case. Fix any $t \in [0,T)$ and let $A = \left\{\theta \in [0,2\pi] \mid k(\theta,t) > k^*(t)\right\}$. We have the estimate on A: \\
\[ \int_0^{2\pi}(a_0k)^2-\left((a_0k)'\right)^2d\theta \leq 2k^*(t)\int_0^{2\pi}a_0(a_0+a_0'')k \ d\theta - k^*(t)^2\int_A(a_0')^2+2a_0a_0''+a_0^2d\theta \] \\
Here we must take some care with the first integral on the right side. First, notice that \\
\begin{eqnarray*} \int_0^{2\pi}a_0(a_0+a_0'')k \ d\theta = \frac{1}{f(t)^2}\int_0^{2\pi}[p(\theta,t),p'(\theta,t)]k\ d\theta = \frac{1}{f(t)^2}\int_0^{2\pi}k^2\lambda \ d\theta = \\ \\ \frac{1}{f(t)^2} \int_0^{2\pi}k^2v \ du \end{eqnarray*} \\
This is not the derivative of $L_{\mathcal{Q}_t}$ with changed sign as in the usual case, but this won't be a problem. Once we have \\
\[ \int_0^{2\pi}k^2v \ du = -\frac{\partial L_{\mathcal{Q}_t}}{\partial t} + \frac{\dot{f}(t)}{f(t)}\int_0^{2\pi}v \ du = -\frac{\partial L_{\mathcal{Q}_t}}{\partial t} + \frac{\dot{f}(t)}{f(t)}L_{\mathcal{Q}_t} \] \\
we can use uniform bounds (remember $f$ is bounded by below by 1) for $f$, $\dot{f}$ and $L_{\mathcal{Q}_t}$ to write \\
\[ \frac{1}{f(t)^2}\int_0^{2\pi}k^2v \ du \leq -C_0\frac{\partial L_{\mathcal{Q}_t}}{\partial t} + C_1 \] \\
for constants $C_0,C_1>0$ which don't depend on $t$. This yields \\
\[ \int_0^{2\pi}(a_0k)^2-\left((a_0k)'\right)^2d\theta \leq -2k^*(t)C_0\frac{\partial L_{\mathcal{Q}_t}}{\partial t} + 2k^*(t)C_1 + 2\pi k^*(t)C_2, \] \\
where $C_2 = \displaystyle\max_{\theta\in [0,2\pi]}\left|(a_0')^2+2a_0a_0'' + a_0^2\right|$. Now, integrating and using, again, a bound for $L_{\mathcal{Q}_t}$ yields the desired estimate for finite time, which is what we want. \\
\end{proof}

\begin{coro} In the same conditions of the proposition the function $t \mapsto \displaystyle\int_0^{2\pi}a_0(a_0+a_0'')\log(k) \ d\theta$ is bounded in $[0,T)$. \\
\end{coro}
\begin{proof} Let $M$ be an upper bound given for $\displaystyle\int_0^{2\pi}a_0(a_0+a_0'')\log\left(\frac{k}{f}\right)\ d\theta$. Then,  \\ 

\[ \int_0^{2\pi}a_0(a_0+a_0'')\log(k) \ d\theta \leq M + \log(f(t))\int_0^{2\pi}a_0(a_0+a_0'') \ d\theta, \] \\
and then we have the desired since $\log(f(t))$ is bounded for finite time. Notice also that we have an obvious lower bound (not necessarily positive): \\
\[ \int_0^{2\pi}a_0(a_0+a_0'')\log\left(k_{\mathrm{MIN}}(0)\right) \ d\theta \leq \int_ 0^{2\pi}a_0(a_0+a_0'')\log(k) \ d\theta \]
for every $t\in [0,T)$.
\end{proof}

\begin{prop} There exists a constant $N \geq 0$ such that \\ 
\[ \int_0^{2\pi}\left((a_0(\theta)k(\theta,t))'\right)^2d\theta \leq \int_0^{2\pi}\left(a_0(\theta)k(\theta,t)\right)^2d\theta + N \]
for every $T \in [0,T)$.
\end{prop}
\begin{proof} We will first consider the function $g:[0,T) \rightarrow \mathbb{R}$ given by : \\
\[ g(t) = \int_0^{2\pi}\left(a_0(\theta)k(\theta,t)\right)^2 - \left((a_0(\theta)k(\theta,t))'\right)^2 + 2a_0(\theta)\left(a_0(\theta)+a_0''(\theta)\right)\log(k(\theta,t))\frac{\dot{f}(t)}{f(t)} \ d\theta \] \\
After some calculations we have that its derivative is given by \\
\[ \frac{dg}{dt} = \int_0^{2\pi}\frac{2a_0(a_0+a_0'')}{k^2}\left(\frac{\partial k}{\partial t}\right)^2 + 2a_0(a_0+a_0'')\log(k)\frac{\partial^2\log(f)}{\partial t^2} \ d\theta \] \\
and then, \\
\begin{eqnarray*} \frac{dg}{dt} \geq 2\int_0^{2\pi}a_0(a_0+a_0'')\log(k)\frac{\partial^2\log(f)}{\partial t}^2\ d\theta \geq \\ \\ \geq  2\frac{\partial^2\log(f)}{\partial t^2}\int_0^{2\pi}a_0(a_0+a_0'')\log\left(k_{\mathrm{MIN}}(0)\right) \ d\theta  \geq C \end{eqnarray*}
for some constant $C$ that we cannot take positive because we may have $k_{\mathrm{MIN}}(0) < 1$. In particular we can assume $C <0$. Notice that we used, again, uniform bounds for $f$ and its derivatives in finite time. By integration we have, for each $t \in [0,T)$ \\
\[ g(t) \geq g(0) + Ct \geq g(0) + CT, \] \\ 
since $C < 0$. Then, substituting $g$ by its expression and rearranging the terms we have for every $t \in [0,T)$,  \\
\begin{eqnarray*} \int_0^{2\pi}\left((a_0k)'\right)^2d\theta \leq -g(0) - CT + \int_0^{2\pi}(a_0k)^2d\theta + 2\frac{\dot{f}(t)}{f(t)}\int_0^{2\pi}a_0(a_0+a_0'')\log k \ d\theta \leq \\ \\ \leq \int_0^{2\pi}(a_0k)^2d\theta -g(0) - CT + M, \end{eqnarray*}
where $M$ is a time-independent constant given by bounds on $f$ and $f'$ and by the previous corollary. Taking $N$ to be any positive number greater than $-g(0) - CT + M$ yields the desired. \\
\end{proof}
\begin{lemma} If $\displaystyle\int_0^{2\pi}a_0(a_0+a_0'')\log\left(k\right) \ d\theta$ is bounded in $[0,T)$, then for any $\delta > 0$ there exists a constant $C$ such that if $k(\theta,t) > C$ in an interval $J$ (varying the parameter $\theta$) then we have necessarily $|J| \leq \delta$. \\
\end{lemma}
\begin{proof} The proof is identical to the proof in the usual case. \\
\end{proof}

\begin{prop} If $\displaystyle\int_0^{2\pi}a_0(a_0+a_0'')\log\left(k\right) \ d\theta$ is bounded in $[0,T)$, then $k(\theta,t)$ has an upper bound in $S^1\times [0,T)$. \\
\end{prop}
\begin{proof} The proof is, again, identical to the proof in the usual case. \\
\end{proof}
Combining these lemmas and propositions yields \\
\begin{teo}If the area $A(t)$ enclosed by the curves associated to the $t$-minkowskian curvature function $k$ admits a strictly positive lower bound on $[0,T)$, then $k$ is uniformly bounded in $S^1\times [0,T)$. \\
 \end{teo}
Let us prove now that the first spatial derivative of $k$ is also bounded provided $k$ is bounded. \\
\begin{prop} If $k$ is bounded in $S^1\times [0,T)$, then $k'$ is also bounded in $S^1\times [0,T)$. \\
\end{prop}
\begin{proof} As in the usual case, consider the function $u:S^1\times [0,T) \rightarrow \mathbb{R}$ given by $u = k'e^{ct+h(\theta)}$, where $h(\theta) = \log\left(a_0(\theta)^2\right)$ and $c$ is to be choosen later. After some calculations we have that $u$ is a solution of the second-order linear parabolic equation: \\
\[ \frac{\partial u}{\partial t} = \left(c + 3k^2 + \frac{\dot{f}(t)}{f(t)}\right)u - k^2\frac{2a_0'}{a_0+a_0''}\frac{\partial u}{\partial\theta} + \frac{\partial}{\partial\theta}\left(k^2\frac{a_0}{a_0+a_0''}\frac{\partial u}{\partial\theta}\right) \] \\
Now, using bounds for $k$, $f$ and $\dot{f}$ one can choose $c$ such that the coeficient of $u$ is nonpositive. Then, using the maximum principle we have that $u$ is bounded in $S^1\times [0,T)$, and then $k'$ is also bounded since we are working with finite time. \\
\end{proof}
To show that the higher derivatives of $k$ are also bounded in $S^1\times [0,T)$ we repeat the proof of the usual case adding some terms given by the new term on the evolution equation. The bounds for $f$ and its derivatives will provide the necessary bounds for this new terms. Thus we have, as in the usual case: \\
\begin{teo}The solution to the generalized minkowskian curvature flow associated to the family of norms given by $a(\theta,t)=f(t)a_0(\theta)$ where $a_0$ is a positive, $\pi$-periodic and $C^{\infty}$ function, and $f$ is a positive, nondecreasing, $C^{\infty}$ function with $f(0) = 1$ continues until the area enclosed by the curves converges to zero. \\
\end{teo}
Since the solution cannot continues after the area goes to zero we have an corollary concerning the blow up of a family of PDE's.\\ 
\begin{coro} Fix a function $g:S^1\rightarrow \mathbb{R}$ which is smooth and $\pi$-periodic, and a smooth and strictly positive function $u_0:S^1 \rightarrow\mathbb{R}$ such that \\
\[ \int_0^{2\pi}\frac{g(\theta)+g''(\theta)}{u_0(\theta)}\sin(\theta)d\theta = \int_0^{2\pi}\frac{g(\theta)+g''(\theta)}{u_0(\theta)}\cos(\theta)d\theta = 0. \] \\
Let $\mathcal{F}$ be the family of the positive, nondecreasing and smooth functions $f:[0,\infty) \rightarrow \mathbb{R}$ with $f(0) = 1$, and consider the associated family of PDE's: \\
\[ \left\{ \begin{array}{lll} u_t = \displaystyle\frac{g}{g+g''}u^2u_{\theta\theta}+\frac{2g'}{g+g''}u^2u_{\theta}+u^3+\frac{\dot{f}}{f}u \\ \\ u(0) = u_0 \\ \end{array}\right. \] \\
Then, we must have an uniform upper bound for the blow up time of the solutions to these PDE's that only depends on $g$ and $u_0$. In other words, there exists a time $T = T(u_0,g)$ such that picking any solution $u$ for some of these PDE's we must have $u$ or some of its derivatives blowing up for some $t \leq T$. The time $T$ is explicitly given by the ratio $T = A/2B$ where $A$ is the area of the curve whose minkowskian curvature with respect to the norm given by the curve $a:\theta \mapsto g(\theta)e_r + g'(\theta)e_\theta$ is $u_0$; and $B$ is the area of the curve $a$. \\
\end{coro}
Even though this might be obvious for someone who has familiarity with nonlinear parabolic PDE's, we think the interest here is that we arrived at this result using, essencially, geometric methods. \\

\end{document}